\documentclass[12pt]{amsart}
\usepackage{graphicx}
\usepackage{color}
\newcommand{\beq}{\begin{eqnarray*}}
\newcommand{\feq}{\end{eqnarray*}}
\newcommand{\beqn}{\begin{eqnarray}}
\newcommand{\feqn}{\end{eqnarray}}

\newtheorem{theorem}{Theorem}[section]
\newtheorem{lemma}[theorem]{Lemma}
\newtheorem{corollary}[theorem]{Corollary}

\theoremstyle{definition}

\theoremstyle{remark}

\numberwithin{equation}{section}

\begin{document}
\title[Thresholds for shock formation in traffic flow models]{Thresholds for shock formation in traffic flow models with Arrhenius look-ahead dynamics}
\author{Yongki Lee and Hailiang Liu}
\address{Department of Mathematics, Iowa State University, Ames, Iowa 50010}
\email{yklee@iastate.edu; hliu@iastate.edu}
\begin{abstract}  We investigate a class of nonlocal conservation laws with the nonlinear advection coupling both  local and nonlocal mechanism, which arises in several applications such as the collective motion of cells  and traffic flows.
 It is proved that the $C^1$ solution regularity of this class of conservation laws will persist at least for a short time. This persistency may continue as long as the solution gradient remains bounded.  Based on this result,  we further identify sub-thresholds for finite time shock formation in traffic flow models with Arrhenius look-ahead dynamics.
\end{abstract}
\maketitle
\section{Introduction}
In this  work we investigate a class of nonlocal conservation laws,
\begin{equation}\label{main}
\left\{
  \begin{array}{ll}
    \partial_t u + \partial_x F(u, \bar{u})=0, & t>0, x\in \mathbb{R}, \hbox{} \\
    u(0,x)=u_0 (x), & x\in \mathbb{R}, \hbox{}
  \end{array}
\right.
\end{equation}
where $u$ is the unknown, $F$ is a given smooth function, and  $\bar u$ is given by
\begin{equation}\label{bk}
\bar{u}(t,x)= (K * u)(t,x) =\int_{\mathbb{R}} K(x-y)u(t, y) \, dy,
\end{equation}
where  $K$ is assumed in $ W^{1,1}(\mathbb{R})$.   The advection couples both local and nonlocal mechanism.  
This class of conservation laws  appears in several applications including  traffic flows \cite{Kurganov, Sopasakis},  the collective motion of biological cells \cite{DS05, MDS08, Perthame}, dispersive water waves \cite{Whitham, Holm, DP99, Liu0},   the radiating gas motion \cite{Ha71, Ro89, LT02} and high-frequency waves in relaxing medium \cite{Hunter, Vak1, Vak2}.


We are interested in the persistence of the $C^1$ solution regularity for (\ref{main}).   As is known that the typical well-posedness result asserts that either a solution of a time-dependent PDE exists for all time or else there is a finite time such that some norm of the solution becomes unbounded as the life span is approached. The natural question is whether there is a critical threshold for the initial data such that the persistence of the $C^1$ solution regularity depends only on crossing such a critical threshold. This concept of critical threshold and associated methodology is originated and developed in a series of papers by Engelberg, Liu and Tadmor \cite{Engelberg, LT02-1, LT03-1} for a class of Euler-Poisson equations.

In this paper we attempt to study such a critical phenomena in (\ref{main}).   $C^1$ solution regularity is shown to persist at least for finite time.  Moreover, such persistency may continue as long as the solution gradient remains bounded.   We also identify sub-thresholds for finite time shock formation in some special traffic flow models, as well as \eqref{main} with one sided interaction kernels. These together partially confirm the critical threshold phenomenon in non-local conservation  laws (\ref{main}).

The traffic flow model that motivated this study is the one with looking ahead relaxation introduced by Sopasakis and Katsoulakis \cite{Sopasakis}:
\begin{equation}\label{traffic}
\left\{
  \begin{array}{ll}
    \partial_t u + \partial_x (u(1-u)e^{-K * u}) =0, & t>0, x \in \mathbb{R}, \\
    u(0,x)=u_0 (x), &   x\in \mathbb{R},\hbox{}
  \end{array}
\right.
\end{equation}
where $u(t,x)$ represents a vehicle density normalized in the interval $[0,1]$ and the relaxation kernel
\begin{equation}\label{kernel_1}
K(r)=\left\{
       \begin{array}{ll}
       \frac{K_0}{\gamma}, & \hbox{if $-\gamma \leq r \leq 0$,}\\
         0, & \hbox{otherwise,}
       \end{array}
     \right.
\end{equation}
is the constant interaction potential, where $\gamma$ is a positive constant proportional to the look-ahead distance and $K_0$ is a positive interaction strength.  We set $K_0 =1$ since in our study this parameter
is not essential.

An improved interaction potential for \eqref{traffic} is introduced in \cite{Kurganov} with
\begin{equation}\label{kernel_2}
K(r)=
\left\{
  \begin{array}{ll}
    \frac{2}{\gamma}\big{(} 1+\frac{r}{\gamma}  \big{)}, & \hbox{$-\gamma \leq r \leq 0$,} \\
    0, & \hbox{otherwise.}
  \end{array}
\right.
\end{equation}
This linear potential is intended to take into account the fact that a car's speed is affected more by nearby vehicles than distant ones.  The authors in \cite{Kurganov} carried out some careful numerical study of  the traffic flow model \eqref{traffic},  through three examples: red light traffic, traffic jam on a busy freeway and a numerical breakdown study.  In the case of a good visibility (large $\gamma$), their numerical studies suggest that  \eqref{traffic} with the modified potential \eqref{kernel_2}  yields solutions that seem to better correspond to reality.

The objective of this article is  therefore twofold : i) to establish local wellposedness of smooth solutions for (\ref{main}); ii) to identify threshold conditions for the finite time shock formation of the traffic flow model \eqref{traffic} subject to two different potentials (\ref{kernel_1}) and (\ref{kernel_2}), respectively.  The finite time shock formation of solutions in traffic flows are understood as congestion formation.

We use $X$ to denote a space $X(\mathbb{R})$ for $X=H^2, W^{1,1}$ and $L^{\infty} + H^2$.  The main results are collectively stated as follows.

\begin{theorem}\label{thm1}
(\textbf{Local existence}) Suppose $F \in C^3(\mathbb{R}, \mathbb{R})$ and $K \in W^{1,1}$. If $u_0 \in L^{\infty} + H^2$, then there exists
$T>0$, depending on the data,
such that \eqref{main} admits a unique solution $u \in C^{1} ([0,T)\times \mathbb{R}).$  Moreover,
 if the maximum life span $T^* < \infty$, then
$$
\lim_{t \rightarrow T^*-} \|\partial_x u(t, \cdot) \|_{L^\infty}=\infty.
$$
\end{theorem}

\begin{theorem}\label{thm2}
Consider \eqref{traffic} with constant potential \eqref{kernel_1}. Suppose that $u_0 \in H^2$ and $0\leq u_0 (x) \leq 1$ for all $x \in \mathbb{R}$. If
\begin{equation}\label{blowup_1}
\sup_{x\in \mathbb{R}}[u' _{0} (x)] > \frac{1}{\gamma} \bigg{(}\frac{1}{2}+ \frac{\sqrt{2}}{4} \cdot \sqrt{3 - \min\big{\{}-1, \ \gamma \cdot \inf_{x \in \mathbb{R}} [u' _{0}(x)] \big{\}} } \bigg{)},
\end{equation}
then $u_x$ must blow up at some finite time.
\end{theorem}

\begin{theorem}\label{thm3}
Consider \eqref{traffic} with linear potential \eqref{kernel_2}. Suppose that $u_0 \in H^2$ and $0\leq u_0 (x) \leq 1$ for all $x \in \mathbb{R}$. If
\begin{equation}\label{blowup_2}
\sup_{x\in \mathbb{R}}[u' _{0} (x)] > \frac{1}{\gamma} \bigg{(} 1+ \frac{1}{2} \cdot \sqrt{6 -  \min\big{\{} -2, \ \gamma \cdot \inf_{x \in \mathbb{R}} [u' _{0}(x)] \big{\}} } \bigg{)},
\end{equation}
then $u_x$ must blow up at some finite time.\\ \\
\end{theorem}

Regarding these results several remarks are in order.\\ \\
i) Our threshold results  in Theorems \ref{thm2} and \ref{thm3} are valid for any $0< \gamma < \infty$.  When the look-ahead distance $\gamma \rightarrow \infty$,  both threshold conditions are reduced to  $\sup_{x\in \mathbb{R}}[u' _{0} (x)] > 0$.  On the other hand,  when $\gamma \rightarrow \infty$,  model \eqref{traffic} is reduced to the classical Lightwill-Whitham-Richards(LWR) model \cite{Lighthill,Richards},
$$\partial_t u + \partial_x (u(1-u))=0.$$   This local model can be verified to have finite time shock formation if initial data  has positive slope $u'_0 >0$ at some point.
Therefore,  the threshold conditions identified are consistent with that of  the LWR model.
\\ \\
ii) In a  recent work \cite{Li} D. Li and T. Li presents several finite time shock formation scenarios of solutions to \eqref{traffic} with \eqref{kernel_1}.  Their approach is to analyze the solutions along two characteristic lines defined by  $0=u(t, X_1 (t))$ and $1=u(t,X_2 (t))$, with which they justified that if there exist two points $ \alpha_1 < \alpha_2 $, such that $u_0 (\alpha_1)=0$ and $u_0 (\alpha_2)=1$, then $u_x$ must blow up at some finite time.   Compare to their result, our shock formation conditions in Theorems \ref{thm2} and \ref{thm3} may be viewed in the perspective of critical thresholds.\\ \\
iii) The shock formation conditions in Theorem \ref{thm2} and \ref{thm3} are consistent with the numerical results obtained in \cite{Kurganov}. Indeed,  a numerical comparison in \cite{Kurganov} of solutions to \eqref{traffic} with \eqref{kernel_1} for $\gamma =0.1$ and $\gamma =1$ indicates that the solution with $\gamma =0.1$ remains smooth, while the solution with $\gamma =1$ seems to contain a shock discontinuity.\\ \\
iv) The threshold in \eqref{blowup_2} is bigger than that in \eqref{blowup_1}. This observation suggests that under certain initial configuration, the traffic flow model with constant interaction potential may develop a congestion formation, while the model with the linear interaction potential may not.  Roughly speaking, it is understood that the drivers with the linear potential are `smarter' than the drivers with the constant  potential.\\ \\
v) For fixed $\gamma>0$,  both \eqref{blowup_1} and \eqref{blowup_2} reflect some balance between $\sup_{x\in \mathbb{R}}[u' _{0} (x)]$ and $\inf_{x\in \mathbb{R}}[u' _{0} (x)]$ for the finite time shock formation:
if the non-positive term $\inf_{x\in \mathbb{R}}[u' _{0} (x)] $ is relatively small, then $\sup_{x\in \mathbb{R}}[u' _{0} (x)]$ needs to be large for the finite time shock formation.
It indicates that not only the car density behind the traffic jam but also the car density ahead of the traffic jam contribute to the formation of congestion.\\ \\

{\color{black}

We now summarize the main arguments in our proofs  to follow.   For the proof of Theorem \ref{thm1},  we apply the Banach fixed-point theorem to the transformation $S$ defined through $v=S(u)$,
 where $v$ is solved from
\begin{equation}
\left\{
  \begin{array}{ll}
    \partial_t v + F_u v_x + F_{\bar{u}} \bar{u}_x =0, & \hbox{} \\
    v(t=0)=u_0. & \hbox{}
  \end{array}
\right.
\end{equation}
We show that there exists $T>0$ depending on initial data such that  the mapping $v =S(u)$ exists and is a contraction. In so showing, detailed estimates of  \emph{non-local} terms
are crucial, and allow us  to track the dependence of $T$ on the initial data.

For the proofs of  Theorem \ref{thm2}-\ref{thm3}, we trace the Lagrangian dynamics of $d:=u_x$,  which can be obtained from the Eulerian formulation:
\begin{equation}\label{ed}
(\partial_t +(1-2u)e^{-\bar{u}}\partial_x)d =e^{-\bar{u}}\big{[} 2d^2 +2(1-2u)\bar{u}_x d - u(1-u)\{\bar{u}_x \}^2 +u(1-u)\bar{u}_{xx}  \big{]}.
\end{equation}
The right hand side  is quadratic in $d$,  the a priori bound $0\leq u\leq 1$ ensures the boundedness of both $u$ and $\bar u_x $ involved in the coefficients.
The key in our approach  is to  bound   the nonlocal term $\bar{u}_{xx}$ in terms of  $M=\sup_{x \in \mathbb{R}} [u_x (x,t)]$ and $N=\inf_{x \in \mathbb{R}} [u_x (x,t)]$.
 This way we are able to obtain weakly coupled differential inequalities for both $M$ and $N$,  which yield the desired sub-thresholds.

From the proofs of  Theorem \ref{thm2}-\ref{thm3} we observe that  the one-sided interaction property of kernels \eqref{kernel_1} and \eqref{kernel_2}  is crucial.  Hence
our threshold analysis for the traffic flow models is applicable to the class of nonlocal conservation laws \eqref{main} under the following assumptions:\\
($H_1$).     $F \in C^3 (\mathbb{R}, \mathbb{R})$, and the kernel $K(r) \in W^{1,1} $ satisfying
$$
K(r)=
\left\{
  \begin{array}{ll}
    Nondecreasing, & \hbox{$r\leq 0$,} \\
    0, & \hbox{$r>0$.}
  \end{array}
\right.
$$
($H_2$).   $F(0, \cdot)=F(m, \cdot)=0$ and
$$F_{uu} < 0 , \ F_{\bar{u}\bar{u}} > 0, \ \ F_{\bar{u}}<0\quad \text{for}\quad u\in [0, m].
$$
The result can be stated as follows.
\begin{theorem}\label{thm4}  Consider \eqref{main} with \eqref{bk} under assumptions  ($H_1$)-($H_2$).  If $u_0 \in H^2$ and $0  \leq u_0 (x) \leq m $ for all $x \in \mathbb{R}$, then there exists a non-increasing function
$\lambda(\cdot)$ such that if
$$\sup_{x \in \mathbb{R}} [u' _{0}(x)] > \lambda(\inf_{x \in \mathbb{R}} [u' _{0}(x)] ),
$$
then $u_x$ must blow up at some finite time.
\end{theorem}
}
We should point out that it was the threshold analysis  for traffic flow models that led us to the thresholds \eqref{blowup_1}, \eqref{blowup_2}  in the first place, which in turn was then extended to the general class (\ref{main}) as summarized in Theorem \ref{thm4}.

We now conclude this section by outlining the rest of the paper. In section 2, we prove local wellposedness for the class of nonlocal conservation laws (\ref{main}). In section 3, we investigate sub-thresholds for nonlocal traffic flow models.   We finally sketch the proof of Theorem \ref{thm4} in the  end of this paper.

\section{Local wellposedness and regularity}
In this section, we study the local well-posedness of  \eqref{main}.  We consider a solution space as $u\in u_0(x) +B^T$, with $B^T:= L^{\infty}([0,T];H^2 _x)$,  which allows $u$ to be non-zero at far field. By transformation
$$
U=u-u_0,
$$
we find the following equation for $U \in B^T$,
$$
U_t +\partial_x F(U+u_0 , \bar{U} + \bar{u}_0)=0.
$$
This lies in the same class as \eqref{main}. With this in mind, from now on, we shall
consider
$$u\in B^{T} := L^{\infty}([0,T];H^2 _x).
$$
We prove the local wellposedness result by the fixed point argument. That is,  we first
define a transformation $S$ as  $v =S(u)$, where $v$ is solved from the following equation
\begin{equation}\label{linearized}
\left\{
  \begin{array}{ll}
    \partial_t v + F_u v_x + F_{\bar{u}} \bar{u}_x =0, & \hbox{} \\
    v(t=0)=u_0, & \hbox{}
  \end{array}
\right.
\end{equation}
and then show this mapping has a fixed point.

We begin by verifying the existence of $v=S(u)$, which is carried out in a series of Lemmata 2.1-2.3.  For simplicity, we take
$$a=F_u \ \ \ and \ \ \ b=-F_{\bar{u}} \bar{u}_x. $$
We bound $a$ and $b$ in terms of $u$ in the following lemma.
\begin{lemma}
Suppose $u \in B^{T}$, $K \in W^{1,1}$. 
Then
\begin{equation}\label{est_a}
\|a_x \|_{H^1} \leq (k (1+ \|K \|_{L^1}))^2 ( 1+ \|u_x \|_{\infty})\|u \|_{H^2}
\end{equation}
and
\begin{equation}\label{est_b}
\|b \|_{H^2} \leq k(1+\|K \|_{L^1})^3 (1+ \|K_x \|_{L^1})(1+\|u_x \|_{\infty})^2 \|u \|_{H^2},
\end{equation}
where $k=k(F)$ is a constant depending on $F$.
In particular, if $\sup_{t \in [0,T]} \|u \|_{H^2} \leq R$, then
$$ \sup_{t\in [0,T]} \| a_x \|_{H^1 } < c_a R^2 \ \ and \ \ \sup_{t\in [0,T]} \|b \|_{H^2 } < c_b R^3, $$
where $c_a = k (1+c_1)(1+ \|K \|_{L^1})^2$, $c_b = k (1+c_1)^2 (1+ \|K \|_{W^{1,1}})^4$ and $c_1$ is an embedding constant.
\end{lemma}

\begin{proof}
We begin with some key inequalities for $\bar{u}$: using $\|w * K \|_{L^2} \leq \|K \|_{L^1} \|w \|_{L^2}$ and $K \in W^{1,1}$ we obtain
\begin{equation}\label{bounds}
\begin{split}
&\|\bar{u}_x \|_{L^2} = \|K * u_x \|_{L^2} \leq \|K \|_{L^1} \|u_x \|_{L^2},\\
&\|\bar{u}_{xx} \|_{L^2} = \|K * u_{xx} \|_{L^2} \leq \|K \|_{L^1} \|u_{xx} \|_{L^2},\\
&\|\bar{u}_{xxx} \|_{L^2} =\|K_x * u_{xx} \|_{L^2} \leq \|K_x \|_{L^1} \|u_{xx} \|_{L^2}
\end{split}
\end{equation}
and
$$\|\bar{u}_x \|_{\infty} \leq \|u_x \|_{\infty} \|K \|_{L^1}.$$
We calculate
\begin{equation}\nonumber
\begin{split}
&a_x = F_{uu} u_x + F_{u \bar{u}} \bar{u}_x, \\
&a_{xx} = F_{uuu} u^2 _x + F_{uu\bar{u}} u_x \bar{u}_x + F_{uu}u_{xx} + F_{u\bar{u}u} u_x \bar{u}_x + F_{u\bar{u}\bar{u}} \bar{u}^2 _{x} + F_{u\bar{u}} \bar{u}_{xx},\\
\end{split}
\end{equation}
so that
\begin{equation}\label{est_a_x}\nonumber
\begin{split}
\|a_x \|_{L^2} &\leq k \|u_x \|_{L^2} + k \|K \|_{L^1} \|u_x \|_{L^2}\\
&\leq k (1 + \|K \|_{L^1}) \|u \|_{H^2}.
\end{split}
\end{equation}
\begin{equation}\label{est_a_xx}
\begin{split}
\|a_{xx} \|_{L^2} &\leq k \bigg{(} \|u_x \|_{\infty} \|u_x \|_{L^2} +  \|u_x \|_{\infty} \|K \|_{L^1} \|u_x \|_{L^2} +  \|u_{xx} \|_{L^2}\\
&\hspace{1cm}+  \|u_x \|_{\infty} \|K \|_{L^1} \|u_x \|_{L^2} +  \|u_x \|_{\infty} \|K \|^2 _{L^1} \|u_x \|_{L^2} +  \|K \|_{L^1} \|u_{xx} \|_{L^2} \bigg{)}\\
&\leq k (1 + \|u_x  \|_{\infty} )(1 + \|K \|_{L^1} )^2 \|u \|_{H^2}.
\end{split}
\end{equation}
These together lead to \eqref{est_a}.

We also calculate,
\begin{equation}\nonumber
\begin{split}
&b_x = -F_{\bar{u}u} u_x \bar{u}_x - F_{\bar{u}\bar{u}}\bar{u}^2 _{x} - F_{\bar{u}}\bar{u}_{xx}, \\
&b_{xx} = -F_{\bar{u}uu} u^2 _x \bar{u}_x - F_{\bar{u}u\bar{u}}u_x \bar{u}^2 _x -F_{\bar{u}u}u_x \bar{u}_x - F_{\bar{u}u}u_x \bar{u}_{xx}  \\
&\hspace{1cm} - F_{\bar{u}\bar{u}u} u_x \bar{u}^2 _x - F_{\bar{u}\bar{u}\bar{u}}\bar{u}^3 _x - 2F_{\bar{u}\bar{u}}\bar{u}_x \bar{u}_{xx} \\
&\hspace{1cm} -F_{\bar{u}u}u_x \bar{u}_{xx} - F_{\bar{u} \bar{u}}\bar{u}_x \bar{u}_{xx} - F_{\bar{u}} \bar{u}_{xxx},
\end{split}
\end{equation}
to obtain
\begin{equation}\label{est_b_L}\nonumber
\|b \|_{L^2} \leq k \|K \|_{L^1} \|u_x \|_{L^2}.
\end{equation}
\begin{equation}\label{est_b_x}\nonumber
\begin{split}
\|b_x \|_{L^2} &\leq k \|u_x \|_{\infty} \|K \|_{L^1} \|u_x \|_{L^2} + k \|u_x \|_{\infty} \|K \|^2 _{L^1} \|u_x \|_{L^2} + k \|K \|_{L^1} \|u_{xx} \|_{L^2}\\
&\leq k \bigg{(} (1 + \|u_x \|_{\infty})(1 + \|K \|_{L^1})^2  \bigg{)}  \|u \|_{H^2}.
\end{split}
\end{equation}

\begin{equation}\label{est_b_xx}\nonumber
\begin{split}
\|b_{xx} \|_{L^2} &\leq k \bigg{(} \|u_x \|^2 _{\infty} \|K \|_{L^1}  +  \|u_x \|^2 _{\infty} \|K \|^2 _{L^1}  +  \|u_x \|_{\infty} \|K \|_{L^1}  + \|u_x \| \|K \|_{L^1} \bigg{)} \|u \|_{H^2}\\
&\hspace{0.5cm}+ k \bigg{(} \|u_x \|^2 _{\infty} \|K \|^2 _{L^1} + \|u_x \|^2 _{\infty} \|K \|^3 _{L^1} + 2 \|u_x \|_{\infty} \|K \|^2 _{L^1} \bigg{)} \|u \|_{H^2}\\
&\hspace{0.5cm}+ k \bigg{(} \|u_x \|_{\infty} \|K \|_{L^1} + \|u_x \|_{\infty} \|K \|^2 _{L^1} + \| K_x\|_{L^1}  \bigg{)} \|u \|_{H^2}.\\
\end{split}
\end{equation}
These estimates  give   \eqref{est_b}.
\end{proof}

\begin{lemma}[A priori estimates]
Suppose $u \in B^T$. A sufficiently smooth solution $v$ of \eqref{linearized} must satisfy the energy estimates
\begin{align} \label{v_L}
\sup_{t \in [0,T]} \|v(\cdot ,t ) \|_{L^2} &\leq \bigg{(} \|u_0 \|_{L^2} + T \cdot \sup_{t \in [0,T]} \|b \|_{L^2}  \bigg{)} \exp \bigg{(} \frac{1}{2} \int^T _{0} \|a_x \|_{\infty} d \tau \bigg{)},  \\ \label{v_H}
\sup_{t \in [0,T]} \|v(\cdot ,t) \|_{H^2} &\leq  \bigg{(} \|u_0 \|_{H^2} + T \cdot \sup_{t \in [0,T]} \|b \|_{H^2}  \bigg{)} \exp \bigg{(} \big{(} \frac{3}{2} + c_1 \big{)} \int^T _{0}   \|a_x \|_{H^1} \, d \tau \bigg{)},
\end{align}
where $c_1$ is an embedding constant.
\end{lemma}
\begin{proof}
Apply $\partial^{l} _{x}$ to the first equation of \eqref{linearized} to obtain,
\begin{equation}\label{diff_l1}
(\partial^{l} _{x} v)_t + a\cdot(\partial^l _x v)_x = h^l,
\end{equation}
where $h^{l} = \partial^{l} _{x}b - \partial^{l} _{x} (av_x) + a(\partial^l _x v)_x$.
Multiplying \eqref{diff_l1} by $\partial^l_x v$ and integrating over $\mathbb{R}$, we obtain,
\begin{equation}\label{el}
\frac{1}{2}\frac{d}{dt}  \int_{\mathbb{R}} (\partial^l_x v)^2 \,dx = \int_{\mathbb{R}} a_x \frac{(\partial^l_x v)^2}{2} + \int_{\mathbb{R}} h^l \cdot (\partial^l_x v) \, dx.
\end{equation}
This with $l=0$ leads to
$$
\frac{d}{dt}\|v\|^2_{L^2}=\int a_x v^2 dx +2 \int bvdx \leq \|a_x \|_\infty\|v\|^2 _{L^2} +2\|b\|_{L^2}\|v\|_{L^2}.
$$
That is
$$\frac{d}{dt}\|v \|_{L^2} \leq \frac{1}{2}\|a_x \|_{\infty}\|v \|_{L^2} + \|b \|_{L^2},$$
which upon integration gives \eqref{v_L}.
Next, summing (\ref{el}) for $l=0, 1, 2$, we obtain
\begin{equation}\label{recycle}
\begin{split}
\frac{1}{2}\frac{d}{dt} \|v \|^2 _{H^2} &= \frac{1}{2}\int_{\mathbb{R}} a_x \cdot \sum^2 _{l=0} (\partial^l _x v)^2 \, dx + \int_{\mathbb{R}} \sum^{2} _{l=0} h^l \cdot (\partial^l _x v) \, dx\\
&= \frac{1}{2} \int_{\mathbb{R}} a_x (v^2 - v^2 _x -3v^2 _{xx}) \, dx - \int_{\mathbb{R}} a_{xx} v_x v_{xx} \, dx  + \int_{\mathbb{R}}( bv + b_x v_x + b_{xx} ) v_{xx} \, dx \\
&\leq \frac{3}{2} \|a_x \|_{\infty} \|v \|^2 _{H^2} + \|v_x \|_{\infty} \|a_{xx} \|_{L^2} \|v_{xx} \|_{L^2} + \|b \|_{H^2} \|v \|_{H^2} \\
&\leq \bigg{(} \frac{3}{2} + c_1  \bigg{)} \|a_{x} \|_{H^1} \|v \|^2 _{H^2} +  \|b \|_{H^2} \|v \|_{H^2}. \\
\end{split}
\end{equation}
Therefore, we obtain
$$\frac{d}{dt} \|v \|_{H^2} \leq \bigg{(}\frac{3}{2}  +c_1 \bigg{)} \|a_x \|_{H^1} \|v \|_{H^2} + \|b \|_{H^2},$$
which upon integration again gives \eqref{v_H}.
\end{proof}

\begin{lemma}
Suppose the initial data $v(x, 0)=u_0 \in H^2$.  Then for each $u \in B^T$, there exists a unique solution $v \in B^{T}$ of \eqref{linearized}.
\end{lemma}
\begin{proof}
Since $\sup_{t\in [0,T]} \| a_x \|_{H^1 _x} < \infty$,
$$\frac{dx}{dt}=a, \ \ x(0)=x_0$$
admits a unique solution $x=x(x_0 , t)$ for each $x_0 \in \mathbb{R}$. Along $x(x_0 ,t)$, \eqref{linearized} reduces to
$$\frac{dv}{dt} =b , \ \ v(0)=u_0 (x_0).$$
Hence $v(x(x_0 ,t),t)=u_0 (x_0) + \int^{t} _{0} b(x(x_0,\tau), \tau) \, d\tau$ and the unique solution for \eqref{linearized} exists.
%
\end{proof}

Now, we are ready to prove Theorem \ref{thm1}.\\

\noindent\textbf{Proof of Theorem \ref{thm1}}:
Let $R$ be any number satisfying $R \geq 2 \|u_0 \|_{H^2}$, we define
\begin{equation}
B^{T} _R := \bigg{\{}\omega\in L^{\infty}([0,T];H^2 ) \, | \, \omega(x,0)\equiv u_0, \, \sup_{t \in [0,T]} \|\omega(\cdot ,t) \|_{H^2} \leq R \bigg{\}}.
\end{equation}
Assume that $u \in B^T _R$, we then have
\begin{equation}\nonumber
\|u(t) \|_{\infty} \leq c_0 R, \quad \|u_x (t) \|_{\infty} \leq c_1 R, \ \ \ 0\leq t \leq T,
\end{equation}
where $c_0$ and $c_1$ are the embedding constants.

We first show that $S$ maps $B^T _R$ into $B^T _R$ for some $T$ small. From \eqref{v_H}, it follows that
\begin{equation}
\begin{split}
\sup_{t \in [0,T]} \|v(\cdot, t) \|_{H^2} &\leq \bigg{(} \frac{R}{2} + T \cdot c_b R^3  \bigg{)} \exp \bigg{(} T\cdot \big{(} \frac{3}{2} +c_1\big{)}c_a R^2   \bigg{)} \\
&\leq R,
\end{split}
\end{equation}
provided $$T \leq T_1:= \frac{1}{3(2+c_1)(c_a + c_b)e R^2 }.$$
Hence,
$$S:B^T _R \rightarrow B^T _R , \ \ \ \forall T\leq T_1. $$

We next show that $S$ is a contraction on $B^T _R$ in the $L^\infty([0, T]; L^2_x)$ norm:
\begin{equation}\label{con}
\sup_{t \in [0,T]} \|S(u_1)-S(u_2) \|_{L^2} \leq \frac{1}{2} \cdot \sup_{t \in [0,T]} \|u_1-u_2\|_{L^2}, \; \forall u_1, u_2\in B^T _R.
\end{equation}
Let $\tilde{v}:=v_1 -v_2=S(u_1)-S(u_2)$,  then difference of \eqref{linearized} for $v_2$ and $v_1$, respectively, leads to
\begin{equation}\label{contraction_10}
\tilde{v}_t + a(u_1)\tilde{v}_x=\tilde b, \ \ \ \tilde v (0,x)=0
\end{equation}
with
\begin{equation}\label{b}
\tilde b=: -\{a(u_1) - a(u_2) \}v_{2x}+b(u_1) - b(u_2).
\end{equation}
Applying \eqref{v_L} we have
\begin{equation}\label{2.16}
\sup_{t\in [0,T]} \|\tilde v \|_{L^2}\leq T\cdot \sup_{t\in [0,T]}\|\tilde b(\cdot , t )\|_{L^2} \exp \bigg{(} \frac{1}{2} \int^T _{0} \|\partial_x a(u_1) \|_{\infty} \, d \tau  \bigg{)} .
\end{equation}
In order to find a time interval such that the contraction property \eqref{con} holds, we need to
estimate $\|\partial_x a(u_1) \|_{\infty}$ and  $\|\tilde b(\cdot , t )\|_{L^2}$.

First we have
\begin{equation}\label{a_x}
\begin{split}
\|\partial_x a(u_1) \|_{\infty} &= \|F_{uu} u_{1x} + F_{u\bar{u}} \bar{u}_{1x} \|_{\infty}\\
&\leq k \big{(} \|u_{1x} \|_{\infty} + \|\bar{u}_{1x} \|_{\infty}  \big{)}\\
&\leq k \big{(} c_1 R + c_1 R \|K \|_{L^1}  \big{)}\\
&=:C_1 R.
\end{split}
\end{equation}
The first term in \eqref{b} is bounded as
\begin{equation}\label{b_1}
\| \{a(u_1) - a(u_2) \}v_{2x}  \|_{L^2} \leq C_1 R \|u_1-u_2 \|_{L^2}.
\end{equation}
This can be seen from the following calculation:

\begin{equation*}
\begin{split}
\| \{a(u_1) - a(u_2) \}v_{2x}  \|_{L^2} &= \| \{ F_{u}(u_1 , \bar{u}_1) - F_{u}(u_2 , \bar{u}_2)  \}v_{2x} \|_{L^2}\\
&\leq  \| \{F_{u}(u_1 , \bar{u}_1) - F_{u} (u_2 , \bar{u}_1) \} v_{2x}\|_{L^2} + \| \{F_{u}(u_2 , \bar{u}_1) - F_{u} (u_2 , \bar{u}_2) \} v_{2x}\|_{L^2}\\
&\leq c_1 R k \bigg{(} \|u_1 -u_2 \|_{L^2} + \|\bar{u}_1 - \bar{u}_2 \|_{L^2} \bigg{)}\\
&\leq k c_1 R (1 + \|K\|_{L^1} )\|u_1-u_2 \|_{L^2}.
\end{split}
\end{equation*}
If we assume $F_{\bar{u}}(0, \cdot)=0$,  then the last term in \eqref{b} has a similar bound:
\begin{equation}\label{b_2}
\begin{split}
 \|b(u_1)-b(u_2)\|_{L^2} \leq C_2 R \|\tilde{u} \|_{L^2}.
\end{split}
\end{equation}
To obtain this bound, we decompose it the following way
$$b(u_1) -b(u_2) = -F_{\bar{u}}(u_1 , \bar{u}_1)\{\bar{u}_{1x} - \bar{u}_{2x} \} -\bar{u}_{2x}\{F_{\bar{u}} (u_1 , \bar{u}_1) - F_{\bar{u}}(u_2 , \bar{u}_1) \} -\bar{u}_{2x}\{F_{\bar{u}} (u_2 , \bar{u}_1) - F_{\bar{u}}(u_2 , \bar{u}_2) \}$$
If we assume $F_{\bar{u}}(0, \cdot)=0$, we have $F_{\bar{u}}(u_1 ,\bar{u}_1) = F_{\bar{u}u}(\xi , \bar{u}_1)u_1$,
\begin{equation*}\nonumber
\begin{split}
\|F_{\bar{u}} (u_1 , \bar{u}_1)\{\bar{u}_{1x} - \bar{u}_{2x} \} \|_{L^2}
&\leq k  \|u_1 \{ \bar{u}_{1x} - \bar{u}_{2x} \} \|_{L^2}\\
&\leq k c_0 R \|\bar{u}_{1x} - \bar{u}_{2x} \|_{L^2}\\
&\leq k c_0 R \|K_x \|_{L^1} \| u_1-u_2 \|_{L^2}.
\end{split}
\end{equation*}
Applying the mean value property to the remaining terms gives that
\begin{equation*}
 \|b(u_1)-b(u_2)\|_{L^2} 
\leq k \{c_0 \|K_x \|_{L^1} +c_1   \|K \|_{L^1} + c_1  \|K \|^2 _{L^1}   \} R \|u_1-u_2 \|_{L^2}.
\end{equation*}
Substituting \eqref{a_x}, \eqref{b_1} and \eqref{b_2} into \eqref{2.16}, we obtain
\begin{equation}
\sup_{t \in [0,T]} \|\tilde{v} \|_{L^2} \leq  (C_1 + C_2)R \cdot T e^{C_1 R \cdot T} \sup_{t \in [0,T]} \|u_1-u_2 \|_{L^2},
\end{equation}
which ensures (\ref{con}) if $T \leq T_2$ with
$$T_2= \frac{1}{2e \cdot (C_1 + C_2)R}.$$
Therefore, for $0<T < T^*$ with
$$
T^* =\min\{T_1 ,T_2 \}=\frac{1}{CR^2}\min\{1, R\},
$$
the map $S$ is a contraction on $B^T _R$ in $L^{\infty}([0,T]; L^2 _x)$ norm and thus possesses a unique fixed point $u$ which is the unique solution of \eqref{main}.\\

Note that without assuming $F_{\bar{u}}(0, \cdot)=0$,
a different  bound than  \eqref{b_2} is obtained
$$\|b(u_1) - b(u_2) \|_{L^2} \leq (C_4 + C_3 R) \|\tilde{u} \|_{L^2},$$
hence $T_2$ satisfying
$$T_2 < \frac{1}{2e\{(C_1 + C_2)R + C_4 \}}$$
still ensures the contraction.    This ends the existence proof.

We prove the second part of Theorem \ref{thm1} through the following corollary:

\begin{corollary}\label{cor}
Let $u$ be the solution obtained in Theorem \ref{thm1} with a maximum life span $[0,T)$. Then
\begin{equation}\label{energy_est_2}
\|u(t, \cdot) \|_{H^2} \leq \|u_0 \|_{H^2} \exp {\bigg{(} }k (1+c_1)(1+ \|K \|_{W^{1,1}})^4 \int^t _{0} (1 + \|u_x \|_{\infty})^2 \, d\tau \bigg{)}, \ 0\leq t <T
\end{equation}
where $c_1$ is the embedding constant.  This infers that only one of the following occurs\\
i) \ $T=\infty$ and $u$ is a global solution;\\
ii) $0 <T <\infty$ and
 $$\lim_{t \rightarrow T-} \|\partial_x u(t , \cdot) \|_{L^\infty} = \infty.$$
\end{corollary}
\begin{proof}
We use again the estimate in \eqref{recycle}, setting $v\equiv u$,
$$\frac{d}{dt} \|u \|_{H^2}  \leq \frac{3}{2}\|a_x \|_{\infty} \|u \|_{H^2} + \|u_x \|_{\infty} \|a_{xx} \|_{L^2} + \|b \|_{H^2}. $$
From $a_x = F_{uu} u_x + F_{u\bar u}  \bar{u}_x$,  it follows that $\|a_x \|_{\infty} \leq c_1 k (1 + \|K \|_{L^1}) \|u \|_{H^2}$. Together with the estimates of $\|a_{xx} \|_{L^2}$ and $\|b \|_{H^2}$ in \eqref{est_a_xx} and \eqref{est_b}, respectively, we obtain
$$\frac{d}{dt} \|u \| _{H^2} \leq k (1+c_1)  (1+\|u_x \|_{\infty})^2 (1+\|K\|_{L^1})^3 (1+ \|K_x \|_{L^1}) \|u \| _{H^2}.$$
Upon integration, we obtain \eqref{energy_est_2}.  The claim in ii) follows from a contradiction argument: If $\lim_{t \rightarrow T-} \|u_x \|_{\infty} <\infty$, it would lead to the boundedness of $\|u\|_{H^2}$. One may therefore extend the solution for some $\tilde T >T$, which contradicts the assumption that $T<\infty$ is a maximal existence interval.
\end{proof}

\section{Sub-thresholds for finite time shock formation}
\subsection{Proof of Theorem \ref{thm2}}
In this subsection, we consider the traffic flow model with Arrhenius look-ahead dynamics:
\begin{equation}\label{main_traff}
\left\{
  \begin{array}{ll}
    \partial_t u + \partial_x (u(1-u)e^{-\bar{u}})=0, & \hbox{} \\
    u(0,x)=u_0 (x), & \hbox{}
  \end{array}
\right.
\end{equation}
where $\bar{u}(t,x)=\frac{1}{\gamma}\int^{x+\gamma} _{x} u(t,y) \, dy$. Here $\gamma>0$ denotes look-ahead distance. In the theory of traffic flow, $u(t,x)$ represents a vehicle density normalized in the interval $[0,1]$. 

We want identify some threshold condition for the shock formation of solutions to (\ref{main_traff}). From Corollary \ref{cor} we know that it suffices to track the dynamics of $u_x$.  Our idea is based on tracing $M(t):=\sup_{x\in \mathbb{R}}[u_x (x,t)]$ and $N(t):=\inf_{x\in \mathbb{R}}[u_x (x,t)]$. The existence and differentiability (in almost everywhere sense) of $M(t)$ and $N(t)$ are proved in \cite{Constantin}.

We also state a useful result, which is proved in \cite{Liu1}.
\begin{lemma}\label{lem3.1}(Lemma 3.1. in \cite{Liu1})
Consider the following quadratic equality for $A(t)$
\begin{equation}\label{liu_inequality}
\frac{dA}{dt}=a(t)(A - b_1 (t))(A - b_2(t)), \ \ \ A(0)=A_0,
\end{equation}
with $a(t)>0$, $b_1 (t) \leq b_2 (t)$ and that $a(t)$, $b_1 (t)$, $b_2 (t)$ are uniformly bounded.\\
i) If $A_0 > \max b_2$, then $A(t)$ will experience a finite time blow-up.\\
ii) If there exists a constant $\bar{b}$ such that
$$b_1 (t) \leq \bar{b} \leq b_2(t),$$
then \eqref{liu_inequality} admits a unique global bounded solution satisfying
$$\min\{A_0, \min b_1 \}  \leq  A(t) \leq \bar{b} ,$$
provided $A_0 \leq \bar{b}$.
\end{lemma}
With this result we obtain the following:
\begin{lemma}\label{lem3.2}
Consider the following quadratic inequality,
\begin{equation}\label{liu_inequality2}
\frac{dB}{dt} \geq a(t)(B - b_1 (t))(B - b_2(t)), \ \ \ B(0)=B_0,
\end{equation}
with $a(t)>0$, $b_1 (t) \leq b_2 (t)$ and that $a(t)$, $b_1 (t)$, $b_2 (t)$ are uniformly bounded.\\
i) If $B_0 > \max b_2$, then $B(t)$ will experience a finite time blow-up.\\
ii) $\min\{B_0 , \min b_1 \} \leq B(t)$, for $t \geq 0$ as long as $B(t)$ remains finite on the time interval $[0,t]$.
\end{lemma}
\begin{proof}
i) \ Subtracting \eqref{liu_inequality} from  \eqref{liu_inequality2} gives
$$\frac{d}{dt} (B-A) \geq a(t)(B-A)(B+A -b_1 -b_2).$$
Integration leads to
\begin{equation}\label{monotone}
(B-A)(t) \geq (B_0 -A_0)\exp\bigg{(}\int^t _{0} a(t)(B+A -b_1 -b_2) \, d \tau \bigg{)}.
\end{equation}
Therefore, $B_0 \geq A_0$ implies $B(t) \geq A(t)$.
For any $B_0 > \max b_2$ set $A_0 = B_0$, then by Lemma \ref{lem3.1}, $A_0$  will lead to a finite time blow-up of $A(t)$. Hence, by \eqref{monotone}, $B(t)$ will experience a finite time blow-up.\\
ii) \ Consider \eqref{liu_inequality}, it is easy to see that $\min\{A_0 , \min b_1 \} \leq A(t)$. Then \eqref{monotone} gives the result.
\end{proof}
We remark that Lemma \ref{lem3.2} remains valid even if the quadratic inequality holds almost everywhere.

Now, we are ready to prove Theorem \ref{thm2}.\\


\noindent\textbf{Proof of Theorem \ref{thm2}.}
Let $d:=u_x$ and apply $\partial_t$ to the first equation of \eqref{main_traff},
\begin{equation}\label{d_eqn}
\begin{split}
\dot{d}&:=(\partial_t +(1-2u)e^{-\bar{u}}\partial_x)d\\
&=e^{-\bar{u}}\bigg{[} 2d^2 +2(1-2u)\bar{u}_x d - u(1-u)\{\bar{u}_x \}^2 +u(1-u)\bar{u}_{xx}  \bigg{]}.
\end{split}
\end{equation}
Define for $t \in [0,T)$,
\begin{equation}
\begin{split}
&M(t):=\sup_{x \in \mathbb{R}}[u_x (t,x)]=d(t, \xi(t)),\\
&N(t):=\inf_{x \in \mathbb{R}}[u_x (t,x)]=d(t, \eta(t)).\\
\end{split}
\end{equation}
The existence of $\xi(t)$ and $\eta(t)$ is justified by Theorem 2.1 in \cite{Constantin}. Furthermore, $M(t) \geq 0$ and $N(t) \leq 0$.
Then, along $(t,\xi(t))$, we have
$$\bar{u}_{xx} = \frac{1}{\gamma}\{u_{x} (\xi+\gamma) - u_{x}(\xi)\} \geq \frac{1}{\gamma} (-M +N),$$
and
\eqref{d_eqn} can be written as,
\begin{equation}\label{eqn_M}
\begin{split}
\dot{M} &= e^{-\bar{u}} \bigg{(} 2M^2 + 2(1-2u)\bar{u}_x M -u(1-u)\{\bar{u}_x \}^2 + u(1-u)\bar{u}_{xx} \bigg{)}\\
& \geq e^{-\bar{u}} \bigg{(} 2M^2 + 2(1-2u)\bar{u}_x M -u(1-u)\{\bar{u}_x \}^2 + u(1-u)\frac{(-M +N)}{\gamma} \bigg{)}.
\end{split}
\end{equation}

And along $(t, \eta(t))$,  we have
$$\bar{u}_{xx} = \frac{1}{\gamma} \{ u_{x}(\eta+\gamma) - u_{x} (\eta) \} \geq 0 ,$$
and
\eqref{d_eqn} can be written as,
\begin{equation}\label{eqn_N}
\begin{split}
\dot{N} &= e^{-\bar{u}} \bigg{(} 2N^2 + 2(1-2u)\bar{u}_x N -u(1-u)\{\bar{u}_x \}^2 + u(1-u)\bar{u}_{xx}\bigg{)}\\
&\geq e^{-\bar{u}} \bigg{(} 2N^2 + 2(1-2u)\bar{u}_x N -u(1-u)\{\bar{u}_x \}^2 \bigg{)}.
\end{split}
\end{equation}
\eqref{eqn_N} can be written as
\begin{equation}\label{Nt}
\dot{N} \geq 2 e^{-\bar{u}}(N - N_1)(N -N_2),
\end{equation}
where $$
N_1 (u , \bar{u}_x) = \frac{-(1-2u)\bar{u}_x  - \sqrt{\{(1-2u)\bar{u}_x \}^2 + 2u(1-u)\bar{u}^2 _x } }{2}
$$ and
$$
N_2 (u, \bar{u}_x) =\frac{-(1-2u)\bar{u}_x  + \sqrt{\{(1-2u)\bar{u}_x \}^2 + 2u(1-u)\bar{u}^2 _x } }{2}.
$$
We note that $N_1 \leq 0 \leq N_2$ because $0 \leq u(t) \leq 1$.
It can be shown later that $N_1$  is uniformly bounded from below,
\begin{equation}\label{N1l}
N_1 \geq -\frac{1}{\gamma}.
\end{equation}
Applying Lemma \ref{lem3.2} (ii) to (\ref{Nt}) with $\min_{0 \leq u \leq 1 , \ |\omega|\leq \frac{1}{\gamma}} N_1 (u, \omega)=-\frac{1}{\gamma}$, we obtain
\begin{equation}\nonumber
N(t) \geq \min\bigg{\{} -\frac{1}{\gamma}, \ N(0)  \bigg{\}}=: \frac{ \tilde{N}_0}{\gamma}.
\end{equation}
Substituting this lower bound into \eqref{eqn_M}, we obtain
\begin{equation*}
\begin{split}
\dot{M} &\geq e^{-\bar{u}} \bigg{(} 2M^2 + \bigg{\{} 2(1-2u)\bar{u}_x -\frac{u(1-u)}{\gamma} \bigg{\}} M -u(1-u) \bar{u}^2 _x + \frac{u(1-u)\tilde{N}_0}{\gamma^2} \bigg{)}.
\end{split}
\end{equation*}
Rewriting of this inequality gives
\begin{equation}\label{eqn_M_1}
\dot{M} \geq 2 e^{-\bar{u}}(M-M_1)(M -M_2),
\end{equation}
where $M_2(\geq M_1)$ is given by
$$M_2  :=\frac{-\{2(1-2u)\bar{u}_x -\frac{u(1-u)}{\gamma} \} + \sqrt{\{2(1-2u)\bar{u}_x -\frac{u(1-u)}{\gamma} \}^2 +8u(1-u)\bar{u}^2 _x -8\frac{u(1-u)\tilde{N}_0}{\gamma^2} } }{4}.$$

We claim that $M_2$ has an uniform upper bound,
\begin{equation}\label{M2}
    M_2 \leq \frac{1}{\gamma} \bigg{[} \frac{1}{2}+ \frac{\sqrt{2}}{4}\cdot \sqrt{3-\tilde{N}_0}  \bigg{]}.
\end{equation}
By Lemma \ref{lem3.2} (i), if
$$
M(0)>\frac{1}{\gamma} \bigg{[} \frac{1}{2}+ \frac{\sqrt{2}}{4}\cdot \sqrt{3-\tilde{N}_0}  \bigg{]},
$$
then $M(t)$ will blow up a finite time. This is exactly the threshold condition as stated in Theorem \ref{thm2}.

To complete our proof we still need to verify both claims (\ref{M2}) and (\ref{N1l}). \\

To verify (\ref{M2}), we set
$$
v:=\gamma \cdot \bar{u}_x = u(x + \gamma) - u(x).
$$
From $0\leq u(t) \leq 1$ it follows that $-1 \leq v \leq 1$.  If suffices to find upper bound for $M_2$ over the set
$$
\Omega:=\{(u,v)\in \mathbb{R}^2 \ | \ 0\leq u \leq 1 , \ \ -1 \leq v \leq 1 \}.
$$
In fact,
\begin{equation}\nonumber
\begin{split}
M_2 &=\frac{-\{2(1-2u)v -u(1-u) \} + \sqrt{\{2(1-2u)v -u(1-u) \}^2 +8u(1-u)(v^2 - \tilde{N}_0) } }{4\gamma}\\
&\leq \frac{1}{4\gamma} \big{[} 2+ \sqrt{ 4 + 2(1 - \tilde{N}_0) } \big{]}. \\
\end{split}
\end{equation}
Here, we use $\max_{(u,v)\in \Omega} \{ -2(1-2u)v +u(1-u) \}=2$ which can be verified easily since the underlying function is linear in $v$ and quadratic in $u$. For the next one, $\max_{(u,v)\in \Omega} \{ 8u(1-u)(v^2 - \tilde{N}_0) \} = 2(1-\tilde{N}_0)$ is used, which is obtained from the upper bound $u(1-u)\leq 1/4$.


Finally, we are left with the verification of \eqref{N1l}. With $v$ defined above, we have
\begin{equation}\nonumber
Q:=\gamma N_1=\frac{-(1-2u)v - \sqrt{\{(1-2u)v \}^2 +2u(1-u)v^2  } }{2}.
\end{equation}
By rearranging,
\begin{equation}
\begin{split}
Q^2 &= \frac{u(1-u)v^2}{2} -Q\cdot(1-2u)v\\
&\leq \frac{u(1-u)v^2}{2} + \epsilon Q^2 + \frac{(1-2u)^2}{4\epsilon}v^2, \ \ \ 0 < \epsilon < 1.
\end{split}
\end{equation}
It follows that
\begin{equation}
\begin{split}
(1-\epsilon)Q^2  &\leq \frac{v^2}{4 \epsilon}\{ (1-2u)^2 + 2 \epsilon u(1-u) \}\\
&\leq \frac{1}{4 \epsilon},
\end{split}
\end{equation}
where the maximum value is achieved at $\partial \Omega$. This gives
$$
Q^2 \leq \frac{1}{4 \epsilon (1-\epsilon)}.
$$
Since $\epsilon$ is arbitrary, we choose $\epsilon =\frac{1}{2}$ to get $Q^2\leq 1$, hence $Q\geq -1$, which gives (\ref{N1l}).


\subsection{Proof of Theorem \ref{thm3}}
We rewrite the traffic flow model (\ref{traffic}) with the linear potential as
\begin{equation}\label{main_traff_2}
    \partial_t u + \partial_x (u(1-u)e^{-\tilde{u}})=0,
\end{equation}
where
\begin{equation}\label{tu}
    \tilde{u}(t,x)=\frac{2}{\gamma}\int^{x+\gamma} _{x} \big{(} 1+ \frac{x-y}{\gamma}  \big{)} u(t,y) \, dy.
\end{equation}
Let $d:=u_x$ and apply $\partial_x$ to \eqref{main_traff_2},
\begin{equation}\label{d_eqn_2}
\begin{split}
\dot{d}&=(\partial_t +(1-2u)e^{-\tilde{u}}\partial_x)d\\
&=e^{-\tilde{u}}\bigg{[} 2d^2 +2(1-2u)\tilde{u}_x d - u(1-u)\{\tilde{u}_x \}^2 +u(1-u)\tilde{u}_{xx}  \bigg{]}.
\end{split}
\end{equation}
Here,
\begin{equation}\label{u_tilde}
\begin{split}
&\tilde{u}_x = -\frac{2}{\gamma}\bigg{\{}u(x) - \frac{1}{\gamma} \int^{x+\gamma} _{x} u(y) \, dy \bigg{\}} = -\frac{2}{\gamma} (u - \bar{u} ),\\
&\tilde{u}_{xx} = -\frac{2}{\gamma} (u_x - \bar{u}_x ),
\end{split}
\end{equation}
where $\bar{u}=\frac{1}{\gamma}\int^{x+\gamma} _{x} u(y) \, dy$ as defined in the previous section.
Define for $t \in [0,T)$,
\begin{equation}
\begin{split}
&M(t):=\sup_{x \in \mathbb{R}}[u_x (t,x)]=d(t, \xi(t)),\\
&N(t):=\inf_{x \in \mathbb{R}}[u_x (t,x)]=d(t, \eta(t)).\\
\end{split}
\end{equation}
The existence of $\xi(t)$ and $\eta(t)$ is justified by Theorem 2.1 in \cite{Constantin}.
Then, along $(t,\xi(t))$, \eqref{d_eqn_2} can be written as,
\begin{equation}\label{eqn_M_2}
\begin{split}
\dot{M} &= e^{-\tilde{u}} \bigg{(} 2M^2 + 2(1-2u)\tilde{u}_x M -u(1-u)\{\tilde{u}_x \}^2 + u(1-u)\tilde{u}_{xx} \bigg{)}\\
& \geq e^{-\bar{u}} \bigg{(} 2M^2 + 2(1-2u)\tilde{u}_x M -u(1-u)\{\tilde{u}_x \}^2 + u(1-u)\frac{2(N-M)}{\gamma} \bigg{)},
\end{split}
\end{equation}
where the last inequality follows from the fact that
\begin{equation}\nonumber
\begin{split}
\tilde{u}_{xx}(t,\xi)=\frac{2}{\gamma}(\bar{u}_x - M) \geq \frac{2}{\gamma}(N-M).
\end{split}
\end{equation}
And along $(t, \eta(t) )$, \eqref{d_eqn_2} can be written as,
\begin{equation}\label{eqn_N_2}
\begin{split}
\dot{N} &= e^{-\tilde{u}} \bigg{(} 2N^2 + 2(1-2u)\tilde{u}_x N -u(1-u)\{\tilde{u}_x \}^2 + u(1-u)\tilde{u}_{xx}\bigg{)}\\
&\geq e^{-\tilde{u}} \bigg{(} 2N^2 + 2(1-2u)\tilde{u}_x N -u(1-u)\{\tilde{u}_x \}^2 \bigg{)},
\end{split}
\end{equation}
where the last inequality follows from the fact that $\tilde{u}_{xx} (t, \eta) = \frac{2}{\gamma}(\bar{u}_x - N) \geq 0$.
\eqref{eqn_N_2} can be written as
\begin{equation}\label{Nd}
\dot{N}\geq 2e^{-\tilde{u}}(N -N_1)(N - N_2),
\end{equation}
where
$$N_1 = \frac{-(1-2u)\tilde{u}_x - \sqrt{\{(1-2u)\tilde{u}_x \}^2 + 2u(1-u)\tilde{u}^2 _x } }{2}
 $$
 and
 $$
 N_2 =  \frac{-(1-2u)\tilde{u}_x + \sqrt{\{(1-2u)\tilde{u}_x \}^2 + 2u(1-u)\tilde{u}^2 _x } }{2}.
 $$
 We note that $N_1 \leq 0 \leq N_2$ because $0\leq u(t) \leq 1$.

By using the fact that $0\leq u \leq 1$, and $-2\leq \gamma \tilde{u}_x \leq 2$, it can be shown that $N_1$ is uniformly bounded from below,
$$N_1 \geq -\frac{2}{\gamma}.$$
The verification of this inequality is similar to the one in the proof (\ref{N1l}), details are omitted.
With the lower bound of $N_1 (t)$, Lemma \ref{lem3.2} (ii) when applied to (\ref{Nd}) gives
\begin{equation}
N(t) \geq \min\bigg{\{} -\frac{2}{\gamma}, \ N(0)  \bigg{\}}=: \frac{ \tilde{N}_0}{\gamma}.
\end{equation}
Substituting this lower bound into \eqref{eqn_M_2}, we obtain
\begin{equation}\label{eqn_M_3}
\begin{split}
\dot{M} &\geq e^{-\tilde{u}} \bigg{[} 2M^2 + \bigg{\{} 2(1-2u)\tilde{u}_x -\frac{2u(1-u)}{\gamma} \bigg{\}} M -u(1-u) \tilde{u}^2 _x + \frac{2u(1-u)\tilde{N}_0}{\gamma^2} \bigg{]}\\
&= 2e^{-\tilde{u}}(M -M_1)(M - M_2).
\end{split}
\end{equation}
In order to apply Lemma \ref{lem3.2} (i) to (\ref{eqn_M_3}), we proceed to find the upper bound of $M_2$($\geq M_1$).  Let $v:=\gamma \cdot \tilde{u}_x = -2(u -\bar{u})$, then from the fact that $0 \leq u, \bar{u} \leq 1$, we know that $-2 \leq v \leq 2$. We also let
$$\Omega:=\{(u,v)\in \mathbb{R}^2 \ | \ 0\leq u \leq 1 , \ \ -2 \leq v \leq 2  \}$$
then $M_2$ and it's upper bound are given by
\begin{equation}
\begin{split}
M_2 &=\frac{-\{2(1-2u)v -2u(1-u) \} + \sqrt{\{2(1-2u)v -2u(1-u) \}^2 +8u(1-u)(v^2 - 2 \tilde{N}_0) } }{4\gamma}\\
&\leq \frac{1}{4\gamma} \bigg{[} 4 + \sqrt{16 + 2(4 -2 \tilde{N}_0)} \bigg{]}. \\
\end{split}
\end{equation}
Here, we use $\max_{(u,v)\in \Omega} \{ -2(1-2u)v +u(1-u) \}=4$ which can be verified easily since the underlying function is linear in $v$ and quadratic in $u$. We also use $u(1-u)\leq \frac{1}{4}$ in bounding the term $8u(1-u)(v^2 -2 \tilde{N}_0)$. Therefore, by Lemma \ref{lem3.2} (i),  if
$$M(0)>\frac{1}{\gamma} \bigg{[} 1+ \frac{1}{2}\cdot \sqrt{6-\tilde{N}_0}  \bigg{]},$$
then $M(t)$ experience a finite time blow up. Hence we obtain the desired result.

{\color{black}
\subsection{Proof of Theorem \ref{thm4}} We only sketch the proof since it is entirely similar to that in the previous sections.  Let $d:=u_x$ and apply $\partial_x$ to the first equation of \eqref{main}  to obtain
\begin{equation}\label{general}
(\partial_t + F_u \cdot \partial_x)d= - F_{uu} d^2 - 2F_{u \bar{u}}\bar{u}_x d - F_{\bar{u}\bar{u}} \bar{u}^2 _x -F_{\bar{u}}\bar{u}_{xx}.
\end{equation}
It can be shown that $0\leq u \leq m$, and therefore
$$
|\bar{u}| \leq m \|K \|_{W^{1,1}}, \quad |\bar{u}_x| \leq m \|K \|_{W^{1,1}}.
$$
To find the bound of $\bar{u}_{xx}$, we define for $t \in [0,T)$,
\begin{equation}
\begin{split}
&M(t):=\sup_{x \in \mathbb{R}}[u_x (t,x)]=d(t, \xi(t)),\\
&N(t):=\inf_{x \in \mathbb{R}}[u_x (t,x)]=d(t, \eta(t)).\\
\end{split}
\end{equation}
From \eqref{bk}, it follows that
$$
\bar{u}_{xx} (t, x) = \int^{0} _{-\infty} K'(z) u_{x}(t, x-z) \, dz   -K(0)u_x (t, x).
$$
Therefore, along $\xi(t)$,
$$K(0)(N-M) \leq \bar{u}_{xx} \leq 0,$$
and \eqref{general} is reduced to
\begin{equation}\label{gm}
\dot{M} \geq -F_{uu} M^2 -2F_{u\bar{u}}\bar{u}_x M -F_{\bar{u}\bar{u}}\bar{u}^2 _{x} -F_{\bar{u}}K(0)(N-M).
\end{equation}
Also, along $\eta(t)$,
$$0 \leq \bar{u}_{xx} \leq K(0)(M-N).$$
and \eqref{general} is reduced to
\begin{equation}\label{gn}
\dot{N} \geq -F_{uu} N^2 -2F_{u \bar{u}}\bar{u}_x N -F_{\bar{u}\bar{u}}\bar{u}^2 _{x}=-F_{uu}(N-N_1)(N-N_2),
\end{equation}
where
$$
N_1(u, \bar u_x)=\frac{F_{u \bar{u}} \bar{u}_x  - \sqrt{ \left(F^2 _{u\bar{u}} - F_{uu}F_{\bar{u}\bar{u}}  \right)\bar{u}^2 _x }}{-F_{uu}}.
$$
 From \eqref{gn} we infer the lower bound  of $N(t)$ as
$$N(t) \geq \min \{N(0), \ \min_{0\leq u\leq m, |v|\leq m\|K\|_{W^{1, 1}}}  N_1(u, v) \}=:\tilde{N_0},$$
Substituting this lower bound into \eqref{gm}, we obtain
\begin{equation}\nonumber
\begin{split}
\dot{M}&\geq -F_{uu} M^2 -2F_{u \bar{u}}\bar{u}_x M - F_{\bar{u}\bar{u}}\bar{u}^2 _x -F_{\bar{u}}K(0)(\tilde{N}_0 - M) \\
&= -F_{uu}(M-M_1)(M-M_2),
\end{split}
\end{equation}
where
$$M_2(u, \bar u_x)=\frac{2F_{u\bar{u}}\bar{u}_x - F_{\bar{u}}K(0) + \sqrt{\{2F_{u\bar{u}}\bar{u}_x -F_{\bar{u}}K(0) \}^2 -4\{F_{uu}F_{\bar{u}\bar{u}}\bar{u}^2 _x + F_{uu}F_{\bar{u}}K(0)\tilde{N}_0 \}}}{-2F_{uu}}.
$$
Therefore, by Lemma \ref{lem3.2} (i),  if
$$M(0)> \max_{0\leq u\leq m, |v|\leq m\|K\|_{W^{1, 1}}} M_2(u, v)=: \lambda(N(0)),$$
then $M(t)$ will blow up in finite time. Hence we obtain the desired result.
}

\section*{Acknowledgments}
This research was supported by the National Science Foundation under Grant DMS 09-07963.

\bigskip

\bibliographystyle{abbrv}

\end{document}